\documentclass{amsart}

\usepackage{amsthm,amssymb,verbatim}
\usepackage{graphicx}
\usepackage{enumerate}

\usepackage{hyperref}
\usepackage[alphabetic, msc-links]{amsrefs}

\newcommand{\xx}{\mathbf{x}}
\newcommand{\vv}{\mathbf{v}}
\newcommand{\ee}{\mathbf{e}}

 \newcommand{\FF}{\mathcal{F}}
 \newcommand{\RR}{\mathbf{R}}  
 \newcommand{\BB}{\mathbf{B}}  
 \renewcommand{\SS}{\mathbf{S}}  

 \newcommand{\eps}{\epsilon}
 \newcommand{\Tan}{\operatorname{Tan}}
 
    \newtheorem{theorem}    {Theorem}     
     
    \newtheorem{lemma}      [theorem]       {Lemma}
    \newtheorem*{corollary}{Corollary}

    \newtheorem*{claim}{Claim}

    \theoremstyle{definition}

    \theoremstyle{definition}

\begin{document}

  \newcommand{\marginpor}[1]{}

\renewcommand{\thesubsection}{\thetheorem}

\title[Subsequent Singularities]{Subsequent Singularities in\\
Mean-Convex Mean Curvature Flow}
\author{Brian White}
\address{Department of Mathematics, Stanford University, Stanford, CA 94305}
\thanks{This research was supported by the National Science Foundation
  under grants~DMS-0406209 and DMS~1105330}
\email{white@math.stanford.edu}
\date{March 8, 2011.  Revised March 30, 2013.}

\begin{abstract}
We use Ilmanen's elliptic regularization to prove that for an initially smooth mean convex hypersurface in $\RR^n$ 
moving by mean curvature flow, the surface is very nearly convex in a spacetime
neighborhood of every singularity.  In particular, the tangent flows are all shrinking spheres or cylinders.
Previously this was known only  (i) for $n\le 7$,
and (ii) for arbitrary $n$ up to the first singular time.
\end{abstract}

\maketitle

\section{Introduction}

Let $\overline{\Omega}$ be a smooth compact $n$-dimensional Riemannian manifold with smooth, nonempty boundary
and let $\Omega$ be its interior.
Suppose that the mean curvature vector of $\partial \Omega$ is at every point a nonnegative
multiple of the inward pointing unit normal, and suppose that no connected component of $\partial \Omega$
is a minimal surface.   Then there is a unique weak solution $t\in [0,\infty)\to M_t$ 
of mean curvature flow with $M_0=\partial \Omega$.   The surfaces $M_t$ 
 for distinct values of $t$ are disjoint.  According to~\cite{White-Size}, the singular set
$X$ has Hausdorff dimension at most $(n-2)$.

Under certain circumstances, one can say much more about the singularities.  Let us say that a singular point $x\in M_t$
has {\bf convex type} provided:
\begin{enumerate}
\item the tangent flows at $x$ are shrinking cylinders, and
\item for each sequence $x_i\in M_{t(i)}$ of regular points that converge to $x$,
\[
\liminf_i \frac{\kappa_1(M_{t(i)}, x_i)}{h(M_{t(i)},x_i)}  \ge 0. \tag{*}
\]
\end{enumerate}
Here $\kappa_1, \kappa_2, \dots, \kappa_{n-1}$ are the principal curvatures, ordered so that
$\kappa_1\le \kappa_2 \le \dots \le \kappa_{n-1}$ and so that $h=\sum_i\kappa_i>0$.

Singularities of convex type have other nice properties.  For example, suppose $p_i\in M_{t(i)}$ is a sequence
of regular points that converges to a singular point of convex type. Translate $M_{t(i)}$ by $-p_i$ and dilate
by the mean curvature at $p_i$ to get a surface $M_i'$.  Then the $M_i'$ converge smoothly, after passing
to a subsequence, to smooth convex hypersurface of $\RR^n$.  (See the corollary to theorem~1 
 in~\cite{white-nature}).

In~\cite{white-nature}*{Theorem~1} the following theorem was 
proved\footnote{Part~\eqref{SummaryTheorem-item2} 
of theorem~\ref{SummaryTheorem} was also proved
by Huisken and Sinestrari~\cite{Huisken-Sinestrari} using different methods.}:
\begin{theorem}\label{SummaryTheorem}
\begin{enumerate}[\upshape(1)]
\item
If $n < 8$, then all the singularities are of convex type.
\item\label{SummaryTheorem-item2}
For any $n$, 
if $\Omega\subset \RR^n$ and if $T$ is the first time
that a singularity occurs, then the singularites of the truncated flow
\[
  t\in [0,T] \mapsto M_t
\]
have convex type.  (In other words, tangent flows to the singularities
at time $T$ are shrinking cylinders, and \thetag{*} holds provided $t_i\le T$.)
\end{enumerate}
\end{theorem}

In this paper, we prove that if $\Omega\subset \RR^n$ (or more generally if $\Omega$ is a subset
of a flat Riemannian manifold), then all of the singularities (not just the first singularities)
have convex type.  This result was announced in~\cite{white-nature}*{p.~667}.

The key to proving Theorem~\ref{SummaryTheorem} was ruling out nonflat minimal cones as blowups
(as tangent flows or more generally as limit flows in the terminology of \cite{white-nature}).
  In dimensions $n<8$, this was done using the fact that there are no nonflat
one-sided minimizing minimal hypercones in $\RR^k$ for $k<8$.

For dimensions $n\ge 8$, a different argument was required to exclude nonflat
minimal cones.   In fact, the 
proof\,\footnote{See appendix~\ref{convex-type-appendix} for more details.}
 given in~\cite{white-nature} shows that a singular point $x$
has convex type provided
$
     \kappa_1/h
$
is bounded below on the regular points in some neighborhood of $x$.
Thus to prove that all of the singularities have convex type, it suffices to prove
for each compact set $K$ that there is a finite lower bound on
$
     \kappa_1/h
$
for the regular points of the flow that lie in $K$.  
This we do in theorem~\ref{MainTheorem} below.   

(In~\cite{white-nature}, such a lower bound on $\kappa_1/h$ was proved only up to and including the first singular time,
and hence only the first time-singularities were proven to have convex type.)

We also prove an analogous result for motion of  hypersurfaces
with boundary, where the motion of the boundary is prescribed and the motion of  the interior
is by mean curvature flow.  (In particular, at each regular point, the normal component
of velocity at each interior point is equal to the mean curvature.)   See \S\ref{BoundarySection}.

\section{Translators}\label{Translators}
Let $M$ be a surface in Euclidean space and $\vv$ be a nonzero vector.
We say that $M$ {\em translates} 
with velocity $\vv$ provided the surfaces $M+ t\vv$ ($t\in \RR$)
satisfy the equation for mean curvature flow, i.e., provided the normal
component of the velocity vector (namely $\vv$) at each point is equal
to the mean curvature at that point:
\[
   H = \vv^\perp. \tag{*}
\]
Ilmanen~\cite{Ilmanen-Elliptic} observed that a surface $M$ translates with velocity $\vv$ if and only if it is stationary
for the weighted area:
\[
    \int_{\xx\in M} e^{\vv\cdot \xx}\,dA(\xx).
\]
(One readily checks that \thetag{*} is the Euler-Lagrange equation
for this functional.)
In other words, $M$ translates with velocity $\vv$ if and only if $M$ is a
minimal surface with respect to the Riemannian metric:
\[
  g_{ij}(x) = e^{2\vv\cdot \xx/m}\delta_{ij},
\]
where $m=\dim(M)$.   

\section{A Maximum Principle}

As discussed in the introduction, we need to prove lower bounds on the quantity 
 $\kappa_1/h$.
For smooth flows, $\kappa_1/h$ satisfies a maximum principle which we describe here.

Let $\Sigma$ be a connected $(n-1)$-manifold with boundary and let
$F: \Sigma\times [a,b]\to \RR^n$ be a one-parameter family of immersion that solve 
the normal motion by mean curvature equation
\[
  \frac{\partial F}{\partial t} = H \tag{*}
\]
where $H$ is mean curvature vector of $F(\Sigma,t)$ or, equivalently, the Laplacian of $F$ with respect the metric on 
$\Sigma$ induced by $F(\cdot,t)$.  
Suppose that the mean curvature is everywhere positive.
Let $\kappa_1, \dots, \kappa_n$ be the principle curvatures with respect to the 
direction of motion, so that $h=\sum_i\kappa_i>0$.  The scale invariant quantity $\kappa_1/h$
is a measure of umbilicity: it less than or equal to $1/(n-1)$, with equality if and only if the point
is umbilic.  

\begin{theorem}\label{MaximumPrinciple}
let $F:\Sigma\times [a,b]\to \RR^n$ be a smooth solution of~\thetag{*} with $h>0$ in 
$\Sigma\times [a,b]$.
If the minimum of $\kappa_1/h$ on $\Sigma\times [a,b]$ occurs at a point $(x,t)$ where $x$ is in the interior
of $\Sigma$ and $t>a$, then $\kappa_1/h$ is constant on $\Sigma\times [a,t]$.
\end{theorem}

See~\cite{white-nature}*{Theorem~3} for the proof.

\begin{corollary}\label{MaxPrinCorollary}
Suppose $M\subset \RR^n$ is a compact smooth, connected, mean-convex hypersurface (with boundary) such that the mean curvature is nowhere $0$ and 
such that $M$ translates with velocity $\vv$, i.e., such that
\[
    H = \vv^\perp,
\]
for some vector $\vv$.  Then the minimum of $\kappa_1/h$ on $M$ 
is attained on $\partial M$.  If it is attained at any interior point, then it is constant on $M$.
\end{corollary}

Of course the flow is $t\in\RR \mapsto M +t\vv$.  The corollary follows from the theorem
because for any point interior point $p$ of $M$, there is an $\eps>0$ such that near
$p$ and for $-\eps<t<\eps$, the flow can be parametrized by an $F$ as in the theorem.

\section{The Main Result}

To state the main theorem, it is convenient to introduce some notation.
Let $\Omega$ be an open subset of $\RR^n$ and let $u:\Omega\to\RR$
be a continuous function.  We think of $u$ as describing a flow of hypersurfaces:
\[
  t\in \RR \mapsto u^{-1}(t).
\]
We will say that  $x\in \Omega$ is a {\bf regular point} of the flow provided $u$ is smooth
in a neighborhood of $x$ and $\nabla u(x)\ne 0$. Otherwise $x$ is a singular point of the flow.
For the flows we will consider, the mean curvature is nonzero at every regular point.
(Indeed, the mean curvature vector will be equal to $\nabla u / |\nabla u|^2$.)  For such a point~$x$, we let
$\kappa_1(u,x)\le \kappa_2(u,x)\le\dots \le \kappa_{n-1}(u,x)$ be the principal curvatures
of the level set $u=u(x)$  with respect to normal direction given by the mean curvature
vector, and we let
\[
   h(u,x) := \kappa_1(u,x) + \dots + \kappa_{n-1}(u,x) > 0
\]
be the scalar mean curvature of the level set $u=u(x)$ at $x$.

\begin{theorem}\label{MainTheorem}
Let $\Omega$ be a bounded open domain in $\RR^n$ with smooth, mean convex boundary.
Let $t\in[0,\infty)\mapsto M(t)$ be the mean curvature flow with $M(0)=\partial \Omega$,
and let $u:\overline{\Omega} \to \RR$ be the function such that $u(x)=t$ for  $x=M(t)$.
(Thus $u(x)$ is the time that at which the moving surface reaches the point $x$.)
Then
\begin{enumerate}[\upshape(1)]
\item The singular set $X$ is a compact subset of $\Omega$ with 
Hausdorff dimension at most $(n-2)$, 
and the spacetime singular set\footnote{Spacetime singular sets
and parabolic Hausdorff dimension are not needed in this paper.
The spacetime singular set is $\{(x,u(x)): x\in X\}$.  
The parabolic Hausdorff dimension of a subset of $\RR^n\times \RR$ is
its Hausdorff dimension with respect to the metric 
  $d((x,t),(x',t'))=\max\{ |x-x'|, |t-t'|^{1/2} \}$.
See~\cite{White-Size}*{p.\ 666}
or~\cite{White-Stratification}*{\S7}.}
has parabolic Hausdorff dimension at most $(n-2)$.
\item If $p$ is a singular point, then
\[
    \liminf_{x\in \Omega\setminus X \to p} \frac{\kappa_1(u,x)}{h(u,x)} \ge 0.
\]
\item The tangent flows to the flow $t\mapsto M(t)$ at singularities are all shrinking spheres and cylinders.
\end{enumerate}
\end{theorem}

As stated in the introduction, a singular point for which (2) and (3) hold is said to have {\bf convex type}.
Thus theorem~\ref{MainTheorem} states that the singular set is small, and that all singular points have convex type.

\begin{proof}
Statement (1) was proved in~\cite{White-Size}.  
As mentioned in the introduction (see also appendix~\ref{convex-type-appendix}), to prove convex type, 
it suffices  to prove 

\begin{claim}
Let $\Omega$ and $u:\overline{\Omega}\to \RR$ be as in the statement of 
the theorem.
Let $K$ be a compact subset of $\Omega$.  Then there is an $\alpha_K\in\RR$ such
that 
\[
   \frac{\kappa_1(u,x)}{h(u,x)} \ge \alpha_K
\]
for all regular points $x\in K$.
\end{claim}

By enlarging $K$, we may assume that the interior of $K$ contains the singular set $X$.
For example, we can let $K=\{x: u(x)\ge\eps\}$
for a sufficiently small $\eps>0$.   

To prove the claim, we use elliptic regularization, as described in~\cite{Ilmanen-Elliptic}.
Let $N_\lambda$ be the surface (integral current or flat chain mod $2$)
in $\overline{\Omega}\times \RR$ that minimizes the functional
\[
  \int_{N_\lambda} e^{-\lambda x_{n+1}}\,dA 
  =\int_{N_\lambda}e^{-\lambda \xx\cdot\ee_{n+1}}\,dA
\]
subject to $\partial N_\lambda = [\partial \Omega]\times\{0\}$.
Then by standard arguments (see appendix~{solition-appendix} for details), 
      $N_\lambda$ is the given by the graph
of a smooth function $f_\lambda:\overline{\Omega}\to[0,\infty)$.  Of course
$f_\lambda$ satisfies the Euler-Lagrange equation:
\begin{equation}\label{euler-lagrange}
       H\cdot \nu = - \lambda \ee_{n+1}\cdot \nu.
\end{equation}
(Alternatively, one can use PDE methods to prove existence of a solution $f_\lambda$
to~\eqref{euler-lagrange} with $0$ boundary values, and then let $N_\lambda$ be the graph of $f_\lambda$.)
Then (as in \S\ref{Translators}), $N_\lambda$ translates with velocity $-\lambda\ee_{n+1}$, so that
\[
  t\in \RR \mapsto (N_\lambda)_t := \text{graph}(f_\lambda - \lambda t)
\]
is a family of surfaces in $\Omega\times \RR$ moving by mean curvature.
In particular, the corollary to
 Theorem~\ref{MaximumPrinciple} applies to $N_\lambda$.

Let
\begin{align*}
&U_\lambda: \Omega\times \RR \to \RR \\
&U_\lambda(x,y) = \frac{f_\lambda(x)-y}{\lambda}
\end{align*}
so that $U_\lambda{}^{-1}(t) = (N_\lambda)_t$ for all $t\in \RR$.

Also, let
\begin{align*}
&U: \Omega\times \RR \to \RR \\
&U(x,y)=u(x).
\end{align*}
Thus $U$ is the time-of-arrival function for the mean curvature flow $t\mapsto M_t\times \RR$.

Note that each $\kappa_i(U_\lambda,(x,y))$ (for $x\in \Omega$ and $y\in \RR$)
 is independent of $y$.
Thus by the corollary to Theorem~\ref{MaximumPrinciple} (applied to the graph of $f_\lambda$ over $K$ or, equivalently to $N_\lambda\cap (K\times\RR)$),
\begin{equation}\label{TranslatorInequality}
\frac{\kappa_1(U_\lambda, (p,0))}{h(U_\lambda,(p,0))}
\ge
\min\left\{ 
\frac{\kappa_1(U_\lambda, (x,0))}{h(U_\lambda,(x,0))} : x\in \partial K \right\}.
\end{equation}

Now as $\lambda\to \infty$, the flows $t\in [0,\infty)\mapsto (N_\lambda)_t$
converge\footnote{At first, one only knows subsequential convergence
to {\em some} flow $t\in[0,\infty)\mapsto M'_t\times\RR$ with $M'_0=M_0$. 
But then the uniqueness (or nonfattening) for mean curvature flow
of mean convex surfaces implies that $M'_t\equiv M_t$ for $t\ge 0$.
Since the limit is independent of the subsequence, we in fact have convergence
and not just subsequential convergence.}
as brakke flows to the flow $t\mapsto M_t\times \RR$.  Also, the functions $U_\lambda$ converge uniformly to the function $U$.

By the local regularity theorem in~\cite{White-Local} 
(or by Brakke's regularity theorem~\cite{Brakke}), 
the convergence $U_\lambda\to U$
is smooth on $\Omega\setminus X$, and thus
\begin{equation}\label{*}
\frac{\kappa_1(U_\lambda, \cdot)}{h(U_\lambda,\cdot)} 
\to
\frac{\kappa_1(U, \cdot)}{h(U,\cdot)}  
\end{equation}
uniformly on compact subsets of $(\Omega\setminus X)\times \RR$.
Consequently~\eqref{TranslatorInequality} implies that
\begin{equation}\label{ElongatedInequality}
\frac{\kappa_1(U, (p,0))}{h(U,(p,0))}
\ge
\min\left\{ 
\frac{\kappa_1(U_\lambda, (x,0))}{h(U_\lambda,(x,0))} : x\in \partial K \right\}
\end{equation}
for $p\in K\setminus X$.

Finally, notice that for any point $q\in \Omega\setminus X$, 
\begin{equation}\label{hU=hu}
    h(U, (q,0)) = h(u, q)
\end{equation}
and 
\begin{equation}\label{kU-ku}
   \kappa_1(U, (q,0)) = 0 \wedge \kappa_1(u,q)
\end{equation}
where $a\wedge b$ denotes the smaller of $a$ and $b$.  (This is because the level set of $U$
through $(p,0)$ is the level set of $u$ through $p$ times $\RR$, and thus the principal
curvatures of the former are the principal curvatures of the latter together with $0$.)

By~\eqref{ElongatedInequality}, \eqref{hU=hu}, and~\eqref{kU-ku},
\[
   \frac{\kappa_1(u,p)}{h(u,p)} \ge 0 \wedge \min_{\partial K} \frac{\kappa_1(u,\cdot)}{h(u,\cdot)}.
\]
This proves the claim, and hence the theorem.
\end{proof}


\section{Mean Curvature Flow with Boundary}\label{BoundarySection}

In this section, we consider motion $t\mapsto M_t$ of a manifold with boundary, the motion of the boundary
being prescribed and the motion of the interior being by mean curvature (i.e, so that the normal component
of the velocity vector is equal to the mean curvature vector.)

We need to make certain assumptions to guarantee that the surfaces in the flow have positive mean curvature
(except possibly at time $0$ and/or at the boundary).  
We let $\Sigma\,(=M_0)$ be the prescribed initial surface,
$t\mapsto \Gamma_t$ be the prescribed motion of the boundary, and 
$\overline{W}$ be a region in which the flow happens.  
The assumptions are:

\begin{enumerate}[\upshape(1)]
\item\label{Wpiecewise}
        $\overline{W}$ is a compact, connected region in a smooth Riemannian manifold
such that $\partial W$ is the union of two smooth, compact manifolds
$\Sigma$ and $\Sigma'$ that are disjoint except along their common boundary
 $\Gamma$.  For convenience we assume that $\Sigma$ is connected.
\item\label{W-convex}
        $W$ is mean convex: the mean curvature
of $(\partial W)\setminus \Gamma$ is at each point a nonzero multiple of the
inward-pointing unit normal, and the angle between the tangent halfplanes
to $\Sigma$ and $\Sigma'$ is at most $\pi$ at every point of $\Gamma$.
\item\label{GammaT} 
         $t\in [0,\infty)  \mapsto \Gamma_t$ is a smooth $1$-parameter family of 
submanifolds of $\Sigma'$ such that $\Gamma_0=\Gamma$.
\item\label{monotonic} 
        The family $t\mapsto \Gamma_t$ is monotonic is the following sense:
for each $T$, $\Gamma_T$ is the boundary of $\Sigma'\setminus \cup_{t\le T}\Gamma_t$.
\item\label{moving}
         If $\Sigma$ is a minimal surface, then $\Gamma_t\ne \Gamma_0$
         for $t>0$.
\item\label{GammaInfinity} 
       As $t\to\infty$, the $\Gamma_t$ converge smoothly to an embedded submanifold
$\Gamma_\infty$ of $\Sigma'$.
\end{enumerate}

Note that the $\Gamma_t$ are allowed to move but are not required to do so: setting $\Gamma_t\equiv \Gamma_0$ is allowed.
The assumption~\eqref{moving} guarantees that the surface starts moving
immediately:  if $M_0=\Sigma$ is a minimal surface,
then the $M_t$ will not move until the boundary $\Gamma_t$ does.
The assumption~\eqref{GammaInfinity} that $\Gamma_t$ converges smoothly as $t\to\infty$ is not really restrictive, since if $t\mapsto \Gamma_t$ satisfies (3) and (4) but not (5), we can choose 
an abitrarily large $T<\infty$ and then redefine $t\mapsto \Gamma_t$ for $t\ge T$ by setting 
$\Gamma_t=\Gamma_T$ for $t\ge T$.  In this way, we get a boundary motion that satisfies the hypotheses~(3), (4), and (5) and that agrees with the original boundary motion up to time $T$.

\begin{theorem}\label{BoundaryExistence}
Let $W$, $\Sigma$, and $t\mapsto \Gamma_t$ satisfy the
 hypotheses~\eqref{Wpiecewise}--\eqref{GammaInfinity} above.
Then there is a unique weak solution $t\in[0,T)\mapsto M_t$ of mean curvature flow
such that $M_0=\Sigma$ and  such that $\partial M_t=\Gamma_t$ for all $t$. 
The surfaces $M_t\setminus \Gamma_t$ are disjoint (for distinct values of $t$), and the 
the function
\begin{align*}
&u: \overline{W}\setminus \Sigma' \to [0,\infty] \\
&u(x) = \begin{cases}
t &\text{if $x\in M_t$}\\
\infty &\text{if $x\notin \cup_{0\le t < \infty}M_t$}
\end{cases}
\end{align*}
is a continuous function.  Each $M_t$ is rectifiable, and the multiplicity-one varifolds
associated to the $M_t$ form a brakke flow.

Furthermore, there is a compact subset (the singular set) $X$ of $W$ with the following properties:
\begin{enumerate}[\upshape(1)]
\item The set $X$ has Hausdorff dimension at most $(n-2)$, and the spacetime singular set
has parabolic Hausdorff dimension at most $(n-2)$.
\item Each $M_t\setminus X$ is a smooth properly embedded submanifold of $\overline{\Omega}\setminus X$ with
boundary $\Gamma_t$.
\item If  $t>0$ and $t(i)\to t$, then $M_{t(i)}$ converges smoothly to $M_t$ away from $M_t\cap X$.
 \item The surfaces $M_t$ converge as $t\to\infty$ to a minimal variety $M_\infty$, 
   and the convergence is smooth away from the singular set $X\cap M_\infty$, which
   has Hausdorff dimension at most $n-8$.
\end{enumerate}
\end{theorem}

Here, ``there is a unique weak solution'' is somewhat informal.
The precise statement is: the level-set flow $t\in[0,\infty)\mapsto M_t$ generated by $\Sigma$ and  $t\mapsto \Gamma_t$ is not fattening. (In other words,
the interior of each $M_t$ is empty.)  See~\cite{White-Topology} for level-set
flow of surfaces with boundary.

\begin{proof}
Existence and uniqueness (i.e., non-fattening) can be proved almost exactly as in the boundariless case.  (More precisely, \cite{White-Topology} shows that the level set flow $t\in[0,\infty)\mapsto M_t$ exists
even without any of our special hypotheses~(1)--(6).  However, those hypotheses
imply nonfattening, just as mean convexity of $M_0$ implies nonfattening
in the boundariless case.)
The fact that nonfattening implies that the multiplicity-one varifolds form
a brakke flow is proved in \cite{Ilmanen-Elliptic}.
The interior regularity and the interior behavior at $t\to\infty$ are proved 
in~\cite{White-Size}.
(The results in~\cite{White-Size} are stated for $M_t$ without boundary, but the 
proofs are local and so also work here away from the boundaries $\Gamma_t$.)

The boundary regularity (the statement that $X$ is a compact subset of $W$ and thus is bounded away
from the $\Gamma_t$'s) is a special case of a very general boundary regular result for mean curvature flow
for hypersurfaces~\cite{White-Boundary}.  
(The general result does not require that $M_0$ be mean convex or that
the boundary motion be monotonic.  It is enough that
$M_0$ is contained in a mean convex region whose boundary contains the $\Gamma_t$'s.)
\end{proof}

\begin{theorem}\label{LowDimensionalBoundaries}
Let $t\mapsto M_t$, $u$ and $X$ be as in theorem~\ref{BoundaryExistence}.
If $n<8$, then the singularities of the flow all have convex type.
\end{theorem}

\begin{proof}
The proof is the same as in the boundariless case~\cite{white-nature}, 
since, for dimensions $n<8$, the arguments are local.
\end{proof}

\begin{theorem}\label{MainBoundaryTheorem}
Let $t\mapsto M_t$, $u$ and $X$ be as in theorem~\ref{BoundaryExistence}.
Suppose that $W$ is an open subset of $\RR^n$ and that
\[
\text{the surface $M_\infty$ is smooth.} \tag{*}
\]
Then the singularities of the flow all have convex type.
\end{theorem}

Although the assumption \thetag{*} is certainly unpleasant, it is not as unpleasant as it first appears.
We will indicate below how it can be by-passed in many situations.

\begin{proof}
Let 
\[
  \Omega = \{x\in W: 0<u(x)<\infty\} = \cup_{0<t<\infty}(M_t\setminus \Gamma_t).
\]
As in the the proof of theorem~\ref{MainTheorem},
 it suffices to prove that for every compact set $K$ in $\Omega$, there
is a $\mu=\mu_K\in\RR$ such
that
\[
   \frac{\kappa_1(u,x)}{h(u,x)} \ge \mu
\]
for every regular point $x\in K$.  

The regularity of $M_\infty$ implies that $M_t$ converges smoothly $M_\infty$ as $t\to\infty$ for 
as $t\to\infty$ and thus that $X$ is a compact subset of $\{x: u(x)<\infty\}$.   
Hence by enlarging $K$, we may assume that its interior contains $X$ and thus that $\partial K$
consists entirely of regular points.

Define $U:\Omega\times \RR\to [0,\infty)$ by
\[
 U(x,y) = u(x).
\]
Then of course 
\[
  t\in [0,\infty)\mapsto \overline{U^{-1}(t)} = M_t\times \RR
\]
is a mean curvature flow in $\RR^n\times \RR$ with moving boundary $\Gamma_t\times \RR$.

Next we use the fact that the flow $t\mapsto M_t$ can be constructed by elliptic regularization.

For $\lambda>0$ and for $0<T<\infty$, let
\[
   S_\lambda^T 
   = 
   (M_0\times\{0\}) \cup (\cup_{t\in [0,T]} (\Gamma_t \times \{\lambda t\}) \cup (M_T\times \{\lambda T\}).
\]
and let 
\[
   S_\lambda = S_\lambda^{T(\lambda)}
\]
where $T(\lambda)$ is any function of $\lambda$ such that $T(\lambda)/\lambda\to \infty$
as $\lambda\to \infty$. (For example, one can let $T(\lambda)=\lambda^2$.)

Note that $S_\lambda$ is only piecewise smooth; it has corners along $\Gamma_0\times\{0\}$
and $\Gamma_{T(\lambda)}\times \{T(\lambda)\}$.  Those corners are slightly inconvenient.
So let $\tilde S_\lambda \subset \partial W$ be a smooth manifold obtained from 
 $S_\lambda$ by
rounding the corners in such a way that:
\begin{enumerate}
\item $\tilde S_\lambda$ and $S_\lambda$ coincide except in a tubular neighborhood
of radius $1/\lambda$ of the corners of $S_\lambda$,
\item The principal curvatures of the $\tilde S_\lambda$ are bounded by $C/\lambda$ for
some constant $C$.
\item $\tilde S_\lambda$ is a generalized graph over $\partial W$ in the following sense:
 for each $x\in \partial W$, the intersection of $\{x\}\times \RR$ with $\tilde S_\lambda$
 consists of either a single point or a line segment.
\end{enumerate}
Then there is a unique surface $N_\lambda$ in $\overline{\Omega\times\RR}$ that minimizes the weighted area
\[
   \int_{N_\lambda} e^{-\lambda x_{n+1}} \, dA(x)
\]
subject to $\partial (N_\lambda)=\tilde S_\lambda$.     
Standard arguments (see the appendix~{solition-appendix} for details)
 show that $N_\lambda$ is in fact a smooth manifold with boundary
(the boundary being $\tilde S_\lambda$), 
and that $N_\lambda\setminus \partial N_\lambda$ is the graph a smooth
function $F_\lambda: W\to [0,\infty)$.

As in the proof of theorem~\ref{MainTheorem}, the surface $N_\lambda$ translates with 
velocity $-\lambda\ee_{n+1}$, so that
\[
  t\in \RR \mapsto N_\lambda-  \lambda t\,\ee_{n+1}
\]
is a mean curvature flow.  The corresponding time-of-arrival function
$U_\lambda: W\times \RR \to \RR$ is
\[
   U_\lambda(x,y) = \frac{F_\lambda(x) - y}{\lambda}.
 \]
Furthermore, $U_\lambda$ converges uniformly to $U$ on compact subsets of 
  $\Omega\times \RR$.
The rest of the proof is exactly the same as the proof of theorem~\ref{MainTheorem}.
\end{proof}

The following corollary suggests that the assumption \thetag{*} is not
as troublesome as it might seem:

\begin{corollary}
Suppose that $W$, $\Gamma_t$, and $M_t$ are as in theorem~\ref{MainBoundaryTheorem}, but
without the requirement that $M_\infty$ be smooth.
Let $T\in (0,\infty)$, and suppose
 that the boundary
motion $t\mapsto \Gamma_t$ can be redefined for $t\ge T$ to
get a smooth, monotonic boundary motion $t\mapsto \Gamma'_t$
such that for the corresponding flow $t\mapsto M'_t$, the limit surface
$M'_\infty$ is smooth.  Then the singularities of the original flow up to time $T$
all have convex type.
\end{corollary}

\begin{proof}
The singularities of the new flow $t\mapsto M'_t$ all have convex type
by theorem~\ref{MainBoundaryTheorem}.  But the old flow coincides with the new one up to time $T$.
\end{proof}

\begin{theorem}\label{BallCase}
Suppose that $W$, $\Sigma=M_0$, $\Gamma_t$, and $M_t$ are as in theorem~\ref{MainBoundaryTheorem}, but
without the requirement that $M_\infty$ be smooth.
Suppose that each connected component of $\Sigma'$ is diffeomorphic to a ball.
Then the singularities of the flow have convex type at all finite times.
\end{theorem}

\begin{proof}
For notational simplicity, we assume that $\Sigma'$ has just one component.
Let $0<T<\infty$.   Note that $\Sigma^*:=\Sigma' \setminus \cup_{0\le t\le T}$ is diffeomorphic to a ball.
Let $p$ be a point in the interior of $\Sigma^*$ and let $U$ be a very small neighborhood of $p$ such
that $\overline{U}$ is contained in the interior of $\Sigma^*$ and such that $\overline{U}$ projects
diffeomorphically onto a convex subset $C$ of $\Tan_p\Sigma^*$.  The fact that 
 $\overline{U}$ projects
diffeomorphically onto  a convex region $C$ of a plane implies that $\partial U$ bounds a smooth
minimal surface $V$ (which is actually a graph over $C$) and no other minimal varieties.

Note that $\Sigma^*\setminus U$ is diffeomorphic to $\BB^{n-2}\times [0,1]$.  Thus the boundary
motion
\[
  t\in [0,T]\mapsto \Gamma_t
\]
can be extended to get a smooth, monotonic boundary motion
\[
  t\in [0,\infty] \mapsto \Gamma'_t
\]
such that $\Gamma'_\infty=\partial U$.  
Since $\Gamma'_\infty$ bounds a smooth
minimal surface $V$ and no other minimal varieties, in fact $M'_\infty=V$,
so by the corollary to theorem~\ref{MainBoundaryTheorem}, the singularities of the flow $t\mapsto M_t$
have convex-type up to time $T$.  Since $T$ is arbitrary, the singularities have 
convex type for all finite times.
\end{proof}

\appendix

\section{Translating Solitions}\label{solition-appendix}

Here we prove existence and regularity of the translating solitons $N_\lambda$ that
were used in the proofs of Theorems~\ref{MainTheorem} 
 and~\ref{MainBoundaryTheorem}.

\begin{lemma}\label{Hzero}
Let $M$ be a connected, oriented hypersurface such that the mean curvature $H$ is everywhere
a nonnegative multiple of the unit normal $\nu$ and such that $M$ translates with constant velocity $\vv\ne 0$
under mean curvature flow, i.e., such that
\[
    H\cdot \nu = \vv\cdot \nu.
\]
Then $H\cdot \nu$ has no local interior minimum unless $H\equiv 0$ (in which case, $M$ lies in $\Sigma\times L$
where $L$ is a line parallel to $\vv$ and where $\Sigma$ is a minimal hypersurface 
in $L^\perp.$)
\end{lemma}

\begin{proof}
This is because under mean curvature flow, $H\cdot \nu$ satisfies a nice second order parabolic equation.
The result then follows immediately from the strong maximum principle.  
See for example~\cite{white-nature}*{Theorem 2}.
\end{proof}

\begin{lemma}\label{avoidance}
Let $M_1$ and $M_2$ be two disjoint, compact, connected smooth hypersurfaces in $\RR^{n+1}$
such that $M_i$ translates by $\vv\ne 0$ under mean curvature flow.
If the function
\[
   (p_1,p_2)\in M_1\times M_2 \mapsto |p-q|
\]
achieves its minimum value $\mu$ at  $(p,q)$ for some interior points $p\in M_1$
and $q\in M_2$, 
then $M_1$ and $M_2$ lie in parallel hyperplanes, and the vector $\vv$ is parallel
to those planes.
\end{lemma}

\begin{proof}
Since the result is essentially local, we may assume that $M_1$ and $M_2$ are graphs
of functions $f_1$ and $f_2$ over the same domain $D$ in a hyperplane perpendicular to the line $p_1p_2$,
with $f_2>f_1$ at all points.  Note that $f_2-f_1$ has an local minimum.  Thus by the maximum principle,
it is constant: $f_2-f_1\equiv \mu$.   Consequently, for each point $x$ in the domain of the $f_i$,
we have
$f_2(x)-f_1(x)=\mu$, the minimum distance between the two graphs.  But that implies that $Df_2(x)=Df_1(x)=0$.
In other words, $f_1$ and $f_2$ must be constant.  Thus $M_1$ and $M_2$ are planar.
The fact that $\vv$ is parallel to the those planes follows immediately
 (by Lemma~\ref{Hzero}, for example.)
\end{proof}

\begin{theorem}
Let $n\ge 2$.
Let $W$ be a bounded region in an $\RR^n$ with piecewise smooth, mean convex boundary and 
let $\lambda\ge 0$.   Let $S$ be a smooth, closed $(n-1)$-manifold in $(\partial W)\times\RR$ that
is graph-like in the following sense: each line $\{x\}\times \RR$ with $x\in\partial W$ intersects
$S$ either in a point or in a line segment.

Then there is a smooth, compact $n$-manifold $N$ in $\overline{W}\times \RR$ such that
\begin{enumerate}
\item $\partial N=S$,
\item $N\setminus S$ is the graph of a smooth function $f: W\to \RR$,
\item $N$ translates with velocity $-\lambda\ee_{n+1}$.
\end{enumerate}
\end{theorem}

\begin{proof}
First we recall that there is an entire function $\phi: \RR^n\to \RR$ such
that $\phi$ is rotationally symmetric (i.e., such that $\phi(x)$ depends only on $|x|$)
and such that the graph of $\phi$ translates with velocity $-\lambda\ee_{n+1}$.
By adding a constant to $\phi$, we can assume that $S$ lies below the graph of $\phi$  
Let $b= \min \{y: (x,y)\in S\}$.  Let $R$ be the region
\[
   \{ (x,y) \in \overline{W}\times \RR:   b \le y \le \phi(x)\}.
\]
Now let $N$ be a variety (e.g. an integral current) such that $N$ minimizes the $\lambda$-area
among all varieties that have boundary $S$ and that lie in the region $R$.

By the maximum principle, $N\setminus S$ lies in the interior of $R$.
By~\cite{Hardt-Simon-Boundary}, $N$ is a smooth manifold with 
boundary in a neighborhood of $S$.
Thus the singular set $X$ is a compact subset of the interior of $R$.
Also, the singular set has Hausdorff dimension at most $n-7$.

Let $N(s)$ be the result of translating $N$ upward by distance $s$.

{\bf Claim 1}: for $s>0$, the interiors of $N$ and $N(s)$ are disjoint.

Proof of claim 1: Suppose not.  Let $s$ be the largest number such that $N$ and $N(s)$
intersect in the interior or are tangent to each other at a common boundary point.
By the boundary maximum principle, they cannot be tangent to each other at a common
boundary point.  Thus they touch at an one or more interior points.  Let $Z$ be the
set of interior points where they touch. Then $Z$ is a compact subset $W\times \RR$.
Now it follows that for sufficiently small $\eps>0$, the function
\[
   (p,q) \in N\times N(s+\eps) \mapsto |p-q|
\]
will have an interior local minimum $(p,q)$.   But Lemma~\ref{avoidance}
 then implies that $N$ and $N(s+\eps)$
lie in parallel vertical planes, which is clearly incompatible with the boundary conditions.
This completes the proof of claim 2.

Claim 1 implies that $N\setminus S$ is the graph of a continuous function $f: W\to \RR$.

It remains to show that the singular set $X$ is empty.

Let $K\subset W$ be a compact set with smooth boundary such that the singular set $X$
is contained in the interior of $K\times \RR$.

By Lemma~\ref{Hzero}, $\Tan_pN$ cannot be vertical at any regular point.  In particular,
this means that the function $f$ is smooth on except on the compact set $\pi(X)$
(where $\pi:\RR^n\times \RR\to \RR^n$ is projection
onto the first factor.)

{\bf Claim 2}: $\sup \{ |Df(x)| : x\in K\setminus \pi(X)\} \le \sup\{|Df(x)|: x\in \partial K\}$.

For suppose not.  Then for every small $\eps>0$, the minimum distance between
the graph of $f|K$ and the graph of $(f+\eps)|K$ is attained at a pair of interior
points $p\in \text{graph}(f|K)$ and $q\in \text{graph}((f+\eps)|K)$.  
Note that the tangent cone to $N$ at $p$ is contained in the 
halfspace $\{v: v\cdot (q-p)\le 0\}$, which implies that the tangent cone
is a plane, which by the standard regularity theory for area minimizing hypersurfaces implies
that $p$ is a regular point.  Likewise $q$ is a regular point.
But now $p$ and $q$ violate Lemma~\ref{avoidance}.  This completes the proof of claim~2.

Now claim 2 implies that $f$ is locally Lipschitz.  But that implies (by standard PDE or by standard
GMT) that $f$ is smooth on $W$.  (For the GMT argument, consider the tangent cone at any point.
By claim 2, the tangent cone is the graph of a lipschitz function.  Thus it suffices to show that there
is no nonplanar minimal cone that is the graph of lipshitz function $u$. By dimension reducing, we
may assume that the cone is smooth away from the origin.
 Let $\vv$ be the unit vector for which $f(\vv)$ is a maximum.  Near $\vv$, the graph of $f$ lies below
 the tangent plane at $(\vv,f(\vv))$.  Hence by the strong maximum principle, $f$ is linear in a neighborhood of $\vv$.
 By analytic continuation, $f$ is linear everywhere.)
 \end{proof}
 
\section{$\kappa_1/h$ bounded below implies convex type}\label{convex-type-appendix}

\begin{theorem}
Suppose $t \in [0,\infty)\mapsto M(t)$ is a (possibly singular) mean curvature flow of mean convex $(n-1)$-dimensional
surfaces in a smooth $n$-dimensional Riemannian manifold.  
Let $Z=(z,t)$ be a spacetime singular point of the flow with $t>0$, and suppose that
\[
   \kappa_1/h  \ge c  > -\infty  \tag{*}
\]
in a space-time neighborhood of $Z$.

Then the singular point $z\in M_t$ has convex type.
\end{theorem}

\begin{proof}
In~\cite{white-nature}, this (and all the results of that paper) are proved under the restriction:
\begin{enumerate}[\upshape (\dag)]
\item
 For $n>7$, we work in $\RR^n$ and we consider the flow only up to and including the first singular time. 
\end{enumerate}
In fact, the {\em only} place that restriction is explicitly used is in the proof of
\cite{white-nature}*{theorem~4}.  
In the proof of that theorem, the restriction~\thetag{\dag} is used only to show that $\kappa_1/h$
is bounded below in a (space-time) neighborhood  of the singularity in question up to and including the time $t$.
Thus, if we assume (as here) the hypothesis~\thetag{*} that $\kappa_1/h$ is bounded below in a neighborhood
of $Z$, then theorem 4, and therefore all the other theorems of \cite{white-nature}, hold at that singularity,
without assuming the restriction~\thetag{\dag}.   In particular, the singularity is of convex type 
by \cite{white-nature}*{theorem~1} and its corollary.

Incidentally, the assumption~\thetag{*} that $\kappa_1/h$ is bounded below in a spacetime neighborhood of the
singularity (both before and {\em after} the singularity) significantly simplifies the proofs 
in~\cite{white-nature} in various places.  
In particular, one can (throughout that paper) work with the class $\FF$  of all blowup flows 
(limit flows) at $Z$
 rather than the class of what are are called there
``special limit flows''.    For example, \cite{white-nature}*{theorem~9} asserts 
that any each flow in $\FF$ is convex for all times $\tau\le 0$. 
But the class of blowup flows at a space-time point is trivially closed under time translation, so it follows
immediately that the restriction $\tau\le 0$ is not necessary.   (In \cite{white-nature}, an additional argument
involving an unpublished reference [W6] was given to remove the restriction $\tau\le 0$.)
\end{proof}


\begin{bibdiv}

\begin{biblist}

\bib{Brakke}{book}{
   author={Brakke, Kenneth A.},
   title={The motion of a surface by its mean curvature},
   series={Mathematical Notes},
   volume={20},
   publisher={Princeton University Press},
   place={Princeton, N.J.},
   date={1978},
   pages={i+252},
   isbn={0-691-08204-9},
   review={\MR{485012 (82c:49035)}},
}

\bib{Hardt-Simon-Boundary}{article}{
   author={Hardt, Robert},
   author={Simon, Leon},
   title={Boundary regularity and embedded solutions for the oriented
   Plateau problem},
   journal={Ann. of Math. (2)},
   volume={110},
   date={1979},
   number={3},
   pages={439--486},
   issn={0003-486X},
   review={\MR{554379 (81i:49031)}},
   doi={10.2307/1971233},
}

\bib{Huisken-Sinestrari}{article}{
   author={Huisken, Gerhard},
   author={Sinestrari, Carlo},
   title={Convexity estimates for mean curvature flow and singularities of
   mean convex surfaces},
   journal={Acta Math.},
   volume={183},
   date={1999},
   number={1},
   pages={45--70},
   issn={0001-5962},
   review={\MR{1719551 (2001c:53094)}},
   doi={10.1007/BF02392946},
}
	

\bib{Ilmanen-Elliptic}{article}{
   author={Ilmanen, Tom},
   title={Elliptic regularization and partial regularity for motion by mean
   curvature},
   journal={Mem. Amer. Math. Soc.},
   volume={108},
   date={1994},
   number={520},
   pages={x+90},
   issn={0065-9266},
   review={\MR{1196160 (95d:49060)}},
}

\bib{White-Topology}{article}{
   author={White, Brian},
   title={The topology of hypersurfaces moving by mean curvature},
   journal={Comm. Anal. Geom.},
   volume={3},
   date={1995},
   number={1-2},
   pages={317--333},
   issn={1019-8385},
   review={\MR{1362655 (96k:58051)}},
}

\bib{White-Stratification}{article}{
   author={White, Brian},
   title={Stratification of minimal surfaces, mean curvature flows, and
   harmonic maps},
   journal={J. Reine Angew. Math.},
   volume={488},
   date={1997},
   pages={1--35},
   issn={0075-4102},
   review={\MR{1465365 (99b:49038)}},
   doi={10.1515/crll.1997.488.1},
}

\bib{White-Size}{article}{
   author={White, Brian},
   title={The size of the singular set in mean curvature flow of mean-convex
   sets},
   journal={J. Amer. Math. Soc.},
   volume={13},
   date={2000},
   number={3},
   pages={665--695 (electronic)},
   issn={0894-0347},
   review={\MR{1758759 (2001j:53098)}},
   doi={10.1090/S0894-0347-00-00338-6},
}

\bib{white-nature}{article}{
   author={White, Brian},
   title={The nature of singularities in mean curvature flow of mean-convex
   sets},
   journal={J. Amer. Math. Soc.},
   volume={16},
   date={2003},
   number={1},
   pages={123--138 (electronic)},
   issn={0894-0347},
   review={\MR{1937202 (2003g:53121)}},
   doi={10.1090/S0894-0347-02-00406-X},
}

\bib{White-Local}{article}{
   author={White, Brian},
   title={A local regularity theorem for mean curvature flow},
   journal={Ann. of Math. (2)},
   volume={161},
   date={2005},
   number={3},
   pages={1487--1519},
   issn={0003-486X},
   review={\MR{2180405 (2006i:53100)}},
   doi={10.4007/annals.2005.161.1487},
}

\bib{White-Boundary}{article}{
 author={White, Brian},
 title={Boundary Regularity for Mean Curvature Flow},
 note={In preparation},
 date={2011},
}




\end{biblist}

\end{bibdiv}
\end{document}